\pdfoutput=1
\documentclass[microtype]{gtpart}
\usepackage{graphicx}
\usepackage[mathscr]{eucal}
\usepackage{amssymb}
\usepackage{enumerate, cite}
\usepackage{xcolor}
\definecolor{darkred}{RGB}{200, 60, 0}
\definecolor{mildblue}{RGB}{0, 100, 250}
\usepackage{hyperref}
\hypersetup{
    colorlinks=true,
    linkcolor=darkred,
    urlcolor=darkred,
    citecolor=mildblue,      
    urlcolor=darkred,
}
\usepackage[nameinlink]{cleveref}

\newcommand{\br}{\mathbb{R}}

\newcommand{\bz}{\mathbb Z}
\newcommand{\bn}{\mathbb N}

\newcommand{\vp}{\varphi}

\newcommand{\ssm}{\smallsetminus}

\DeclareMathOperator{\Homeo}{Homeo}

\DeclareMathOperator{\Diff}{Diff}

\providecommand{\co}{\colon\thinspace}

\newtheorem{Thm}{Theorem}[section]
\newtheorem{Thm*}{Theorem}
\newtheorem{Prop}[Thm]{Proposition}
\newtheorem{Lem}[Thm]{Lemma}
\newtheorem{Cor}[Thm]{Corollary}

\newtheorem{MainThm}{Theorem}
\newtheorem{Cor*}[MainThm]{Corollary}

\theoremstyle{definition}
\newtheorem{Def}[Thm]{Definition}

\newtheorem*{Ex*}{Examples}

\newtheorem*{Problem}{Problem}

\numberwithin{equation}{section}

  \usepackage{todonotes} 

\title[Homeomorphisms of Euclidean space are commutators]{Orientation-preserving homeomorphisms of\\Euclidean space are commutators}

\author{Megha Bhat}
\address{Department of Mathematics \\ CUNY Graduate Center \\ New York, NY 10016}
\email{mbhat@gradcenter.cuny.edu}

\author{Nicholas G. Vlamis}
\address{Department of Mathematics \\ CUNY Graduate Center \\ New York, NY 10016, and \newline Department of Mathematics \\ CUNY Queens College \\ Flushing, NY 11367}
\email{nvlamis@gc.cuny.edu}

\begin{document}  

\begin{abstract}
We prove that every orientation-preserving homeomorphism of Euclidean space can be expressed as a commutator of two orientation-preserving homeomorphisms. 
We give an analogous result for annuli.
In the annulus case, we also extend the result to the smooth category in the dimensions for which the associated sphere has a unique smooth structure. 
As a corollary, we establish that every orientation-preserving diffeomorphism of the real line is the commutator of two orientation-preserving diffeomorphisms. 
\end{abstract}

\maketitle



\section{Introduction}
In 1951, Ore \cite{OreSome} initiated the investigation of groups in which every element can be expressed as a commutator.
In particular, he proved that this holds for finite alternating groups and the symmetric group on \( \bn \).
Much more recently, for each \( n \in \bn \), Tsuboi \cite{TsuboiHomeomorphism} showed that the group of orientation-preserving homeomorphisms of the \( n \)-sphere \( \mathbb S^n \) has this property.
At roughly the same time, Basmajian--Maskit \cite{BasmajianSpace} proved that every orientation-preserving isometry of \( \mathbb S^{n-1} \), \( \br^n \), and \( \mathbb H^n \) can be expressed as a commutator for \( n \geq 3 \), where \( \mathbb H^n \) is hyperbolic \( n \)-space.

More than 50 years earlier, Anderson \cite{AndersonAlgebraic} showed that each element of \( \Homeo^+(\mathbb S^2) \) and \( \Homeo^+(\mathbb S^3) \) can be expressed as the product of two commutators. 
Using the generalized Sch\"onflies theorem and the annulus theorem (see the preliminaries below), the techniques used by Anderson in dimension two can be extended to show that, for all \( n \in \bn \), every element of \( \Homeo^+(\mathbb S^n) \) can be expressed as a product of two commutators. 
The same result holds for \( \Homeo^+(\br^n) \) and \( \Homeo_0(\mathbb S^{n-1} \times \br) \) (a proof for \( n = 2 \), using ideas of Le Roux--Mann \cite{LeRouxStrong}, can be found in \cite{VlamisHomeomorphism}; the same proof technique can be used to establish the result in the higher dimensional cases as well). 
Above, \( \Homeo_0(M) \) denotes the connected component of the identity of \( \Homeo(M) \) equipped with the compact-open topology.
We note that \( \Homeo_0(\mathbb S^n) \) coincides with \( \Homeo^+(\mathbb S^n) \),  \( \Homeo_0(\br^n) \)  with  \( \Homeo^+(\br^n) \), and \( \Homeo_0(\mathbb S^n \times \br) \) with the subgroup of \( \Homeo^+(\mathbb S^n \times \br) \) stabilizing each end of \( \mathbb S^n \times \br \) (see the preliminaries below). 

Given this history and  the work of Tsuboi and Basmajian--Maskit, it is natural to ask if every element of \( \Homeo_0(\br^n) \) and \( \Homeo_0(\mathbb S^n \times \br) \) can be expressed as a commutator: our theorem answers this in the affirmative.

\begin{MainThm}
For each \( n \in \bn \), every element of \( \Homeo_0(\mathbb R^n) \) and \( \Homeo_0(\mathbb S^n \times \br) \) can be expressed as a single commutator. 
\end{MainThm}

The proof of our theorem uses the same philosophy as Tsuboi's argument in the spherical case: 
A group element \( f \) can be expressed as a commutator if and only if there exists a group element \( g \) such that \( gf \) and \( g \) are conjugate. 
Tsuboi's idea is to start with an orientation-preserving homeomorphism \( f \) of a sphere and construct a homeomorphism \( g \) having strong enough hyperbolic dynamics with respect to \( f \) so that \( g\circ f \) exhibits the same dynamics as \( g \).
He then uses this dynamical picture to guarantee \( g\circ f \) and \( g \) are conjugate.

The discussion above is a particular instance of a more general phenomenon.
A \emph{word} is an element in a finite-rank free group.
Given \( r \in \bn \) and a word \( {w \in \bz x_1 * \bz x_2 *\cdots*\bz x_r} \), we write \( w = w(x_1, x_2, \ldots, x_r) \) and view \( w \) as an expression in the variables \( x_1, \ldots, x_r \).
Then, given a group \( G \), we have a substitution map \( w \co G^r \to G \) given by \( (g_1, \ldots, g_r) \) maps to \( w(g_1, \ldots, g_r) \) in \( G \). 
Let \( w(G) \) be the subgroup of \( G \) generated by the set \( {G_w = \{ w(\mathbf g), w(\mathbf g)^{-1} : \mathbf g \in G^r \}} \).
Then, \( w(G) \) is a normal subgroup of \( G \) (in fact, it is characteristic).
The \emph{\( w \)-width} of \( G \) is the smallest natural number \( p \) such that every element of \( w(G) \) can be expressed as a product of \( p \) elements in \( G_w \); if such a number does not exist then the \( w \)-width is said to be infinite (see \cite{SegalWords} for more details). 
For example, if \( w = xyx^{-1}y^{-1} \in \bz x * \bz y \), then the \( w \)-width of a group is known as its \emph{commutator width}. 
In this language, Tsuboi's theorem implies that \( \Homeo_0(\mathbb S^n) \) has commutator width one, and our theorem implies   \( \Homeo_0(\br^n) \) and \( \Homeo_0(\mathbb S^n \times \br ) \) also have commutator width one. 

A group \( G \) is \emph{uniformly simple} if there exists \( k \in \bn \) such that for all \( g, f \in G \) with \( g \) nontrivial, \( f \) can be expressed as the product of at most \( k \) conjugates of \( g \) and \( g^{-1} \).
Anderson \cite{AndersonAlgebraic} showed that \( \Homeo_0(\mathbb S^2) \) is uniformly simple, and his argument can be extended to other dimensions using the generalized Sch\"onflies theorem and the annulus theorem; in fact, he showed that \( k \) can be chosen to be eight.
Therefore, given any nontrivial word \( w \), \( w(\Homeo_0(\mathbb S^n)) = \Homeo_0(\mathbb S^n) \) and the \( w \)-width of \( \Homeo_0(\mathbb S^n) \) is at most eight. 
In light of Tsuboi's result, it is natural to ask for which words \( \Homeo_0(\mathbb S^n) \) has width one. 
This can also be asked for \( \Homeo_0(\br^n) \) and \( \Homeo_0(\mathbb S^n \times \br) \); but, unlike \( \Homeo(\mathbb S^n) \), neither \( \Homeo_0(\br^n) \) nor \( \Homeo_0(\mathbb S^n \times \br) \) are simple, as the subgroup of compactly supported homeomorphisms is a nontrivial proper normal subgroup.
However, the germ at the end of \( \br^n \) (resp., at an end of \( \mathbb S^n \times \br \)) is uniformly simple; this first appears in the (unpublished) thesis of Ling \cite{LingAlgebraic} (see also \cite{MannAutomatic} and \cite{SchweitzerNormal}).

\begin{Problem}
Characterize the words for which the word width of \( \Homeo_0(\mathbb S^n) \) (resp., \( \Homeo_0(\br^n) \), \( \Homeo_0(\mathbb S^n \times \br) \)) is equal to one. 
\end{Problem}

As a corollary to our theorem, we show that the \( w \)-width is one for a particular class of words.

\begin{Cor*}
\label{cor:2}
Let \( n, r \in \bn \) with \( r > 1 \), and let \( G \) denote any one of \( \Homeo_0(\mathbb S^n) \), \( \Homeo_0(\br^n) \), or \( \Homeo_0(\mathbb S^n \times \br) \).
If \( p_1, p_2, \ldots, p_r \in \bz \ssm \{0\} \) and \( g \in G \), then there exists \( g_1, g_2, \ldots, g_r \in G \ssm \{1\} \) such that \( g = \prod_{i=1}^r g_i^{\circ p_i} \). 
\end{Cor*}

The proof of \Cref{cor:2} will be given at the end of the note.

\subsection*{A remark on diffeomorphisms}
Here, we provide an accounting of the extent to which our methods extend to the smooth category. 
For a smooth manifold \( M \), let \( \Diff(M) \) denote the group of diffeomorphisms \( M \to M \), and let \( \Diff_0(M) \) denote the connected component of the identity in \( \Diff(M) \) equipped with the compact-open \( C^\infty \)-topology.
Thurston \cite{ThurstonFoliations} proved that  if \( M \) is closed then \( \Diff_0(M) \) is perfect (see \cite{MannShort} for a short proof). 
Burago--Ivanov--Polterovich \cite{BuragoConjugation} showed that  \( \Diff_0(\mathbb S^n) \) is uniformly perfect with commutator width bounded above by four; Rybicki \cite{RybickiBoundedness} proved the same for a class of open manifolds, including \( \Diff_0(\br^n) \) and \( \Diff_0(\mathbb S^n \times \br) \).
Given this history and the fact that \( \Homeo_0( \mathbb S^n ) \), \( \Homeo_0(\br^n) \), and \( \Homeo_0(\mathbb S^n \times \br) \) all have commutator width one, it is natural to ask if the same holds for diffeomorphisms. 

In the case of annuli, we can give a positive answer to this question in specific dimensions, namely in the dimensions \( n \) for which the \( n \)-sphere is known to have a unique smooth structure (with the exception of \( n = 5 \)). 
In these dimensions, as we will now explain, the  preliminary results---as pertain to annuli---presented in Section~\ref{sec:prelim} can be extended to the smooth category.
Moreover, for annuli in these dimensions,  the arguments throughout the note go through verbatim.

Let \( \mathbb B^n \) be the closed unit ball in \( \br^n \), and let \( \Diff_\partial(\mathbb  B^n) \) be the group of diffeomorphisms \( \mathbb B^n \to \mathbb B^n \) fixing a neighborhood of \( \partial \mathbb B^n \) pointwise. 
A point of weakness in extending our results to the smooth setting is \Cref{thm:alexander}, which does not extend in all dimensions.  
However, the proof of \Cref{thm:alexander} given below is valid  in the smooth category for  the dimensions \( n \) in which \( \Diff_\partial(\mathbb B^{n})  \) is connected.
When \( n \in \{1, 2, 3 \} \), it is known that \( \Diff_\partial(\mathbb B^n) \) is connected, and when \( n \geq 5 \), the number of components of \( \Diff_\partial(\mathbb B^{n}) \) is equal to the number of exotic spheres in dimension \( n \), which is encoded by the cardinality of the group \( \Theta_{n+1} \) of h-cobordism classes of homotopy \( (n+1) \)-spheres; we refer the reader to the historical remarks section of \cite{KupersSome} for a more detailed discussion and for references. 

The other results from Section~\ref{sec:prelim} can be adapted using the h-cobordism theorem \cite{SmaleStructure}, which puts an additional dimension restriction, as it is not known whether the h-cobordism theorem holds in dimension three (it fails in general in dimension four, but this turns out not to be relevant to this note). 
We refer the reader to Milnor's notes for details on the h-cobordism theorem and its applications, specifically \cite[\S9 and Concluding~Remarks]{MilnorLectures}.
These are the only dimensional obstructions that arise in extending our arguments for annuli to the smooth setting. 
Accounting for these restrictions allows us to record the following theorem. 

\begin{MainThm}
Let \( n \in \bn \ssm \{4,5\} \).
If \( \Theta_n \) is trivial, then \( \Diff_0( \mathbb S^{n-1} \times \br ) \) has commutator width one. 
\qed
\end{MainThm}

Appealing to \cite[Corollary~1.15]{WangTriviality} for the known values of \( n \in \{ 1, \ldots, 61\} \) in which \( |\Theta_n| = 1 \), we have:

\begin{Cor*}
If \( n \in \{ 1, 2, 3, 6, 12, 56, 61  \} \), then \( \Diff_0( \mathbb S^{n-1} \times \br ) \) has commutator width one.
\qed
\end{Cor*}

In the above corollary, setting \( n = 1 \), we obtain that \( \Diff_0(\mathbb S^0 \times \br) \cong \Diff_0(\br) \times \Diff_0(\br) \) has commutator width one, yielding the following corollary.

\begin{Cor*}
\( \Diff_0(\br) \) has commutator width one. \qed
\end{Cor*}

Our proofs as written do not immediately extend to the case of \( \br^n \), and tracing through the arguments, one finds that the issue can be reduced to a question of differentiability at a single point in the proofs of \Cref{cor:conjugate} and \Cref{thm:rn}. 
For annuli, these issues are pushed off to infinity, allowing us to extend to the smooth category.

\subsection*{Acknowledgements}
The authors thank the anonymous referee for their comments and for suggesting we add a discussion of diffeomorphisms, which did not exist in the original draft.
The second author is supported by NSF DMS-2212922 and PSC-CUNY Awards \#65331-00 53 and \#66435-00 54

\section{Preliminaries}
\label{sec:prelim}

Before we begin, we will need the generalized Sch\"onflies theorem, the annulus theorem, and the characterizations of \( \Homeo_0(\mathbb S^n)\), \( \Homeo_0(\mathbb R^n) \), and \( \Homeo_0(\mathbb S^n \times \br) \)  given in the introduction.
An \( (n-1) \)-dimensional submanifold \( N \) of an \( n \)-manifold \( M \) is \emph{locally flat} if each point of \( N \) has an open neighborhood \( U \) in \( M \) such that the pair \( (U, U \cap N) \) is homeomorphic to \( (\br^n, \br^{n-1}) \).
Additionally, if  \( X \subset M \) is a closed subset with nonempty interior, then we say \( X \) is \emph{locally flat} if \( \partial X \) is a locally flat \( (n-1) \)-dimensional submanifold of \( M \).
In an \( n \)-manifold \( M \), we use the terminology \emph{locally flat annulus} to refer to a locally flat closed subset of \( M \) that is homeomorphic to \( \mathbb S^{n-1} \times [0,1] \).

\begin{Thm}[Generalized Sch\"onflies Theorem {\cite{Brown1,Brown2}}]
If \( \Sigma \) is a locally flat {\( (n-1) \)-dimensional} sphere in \( \mathbb S^n \), then the closure of each component of \( \mathbb S^n \ssm \Sigma \) is homeomorphic to the closed \( n \)-ball. 
\qed
\end{Thm}

\begin{Thm}[Annulus Theorem {\cite{Kirby,Quinn}}]
The closure of the region co-bounded by two disjoint locally flat \( (n-1) \)-dimensional spheres in \( \mathbb S^n \) is a locally flat annulus.
\qed
\end{Thm}

A homeomorphism of \( \br^n \) is \emph{stable} if it can be factored as a composition of homeomorphisms each of which restricts to the identity on an open subset of \( \br^n \). 
The annulus theorem is equivalent to the stable homeomorphism theorem, which states that every homeomorphism of \( \br^n \) is stable \cite{BrownStable}.
From this, together with Alexander's trick, one readily deduces that every orientation-preserving homeomorphism of \( \mathbb R^n \) (resp., \( \mathbb S^n \)) is isotopic to the identity.  
It is also possible to deduce from the stable homeomorphism theorem that every orientation-preserving homeomorphism of \( \mathbb S^n \times \br \) that stabilizes the topological ends is isotopic to the identity; however, it is easier to see this fact by using  the fragmentation lemma (such a proof can be found for \( n =1 \) in \cite{VlamisHomeomorphism}, which can be generalized to higher dimensions).
The fragmentation lemma is a stronger version of the stable homeomorphism theorem that gives control of the open sets being fixed by each homeomorphism in the factorization; it is deduced from the work of Edwards--Kirby \cite{EdwardsDeformations}\footnote{The authors learned of the fragmentation lemma from \cite{MannAutomatic}. We also note, for the sake of extending arguments to the smooth setting, that the fragmentation lemma for diffeomorphisms is an exercise (see \cite[Lemma~2.1]{MannShort}).}.
We record these facts in the following statement.

\begin{Thm}
\label{thm:alexander}
For \( n \in \bn \), every orientation-preserving homeomorphism of \( \mathbb S^n \) (resp., \( \br^n \)) is isotopic to the identity, and every orientation-preserving homeomorphism of \( \mathbb S^n \times \br \) stabilizing the ends is isotopic to the identity. 
\end{Thm}

The following lemma is required to glue together homeomorphisms defined on disjoint pieces of a sphere, especially in the change of coordinates corollary below and the construction of the conjugating map in \Cref{prop:conjugate}.
Before stating the lemma, we need a notion of orientation compatibility for disjoint embeddings of spheres.
 
Let \( \iota_1, \iota_2 \co \mathbb S^{n-1} \to \mathbb S^n \) be embeddings with disjoint locally flat images.
Then the annulus theorem implies that there exists a homeomorphism \( \vp \co \mathbb S^{n-1} \times [0,1] \to \mathbb S^n \) such that \( \vp(x,0) = \iota_1(x) \). 
Let \( \iota\co \mathbb S^{n-1} \to \mathbb S^n \) be given by \( \iota(x) = \vp(x,1) \).
Then \[ \iota^{-1}\circ \iota_2 \co \mathbb S^{n-1} \to \mathbb S^{n-1} \] is a homeomorphism.
We say \( \iota_1 \) and \( \iota_2 \) are \emph{compatibly oriented} if \( \iota^{-1} \circ \iota_2 \) is orientation preserving.

\begin{Lem}
\label{lem:isotopy}
Let \( \iota_1, \iota_2 \co \mathbb S^{n-1} \to \mathbb S^{n} \) be embeddings such that \( \Sigma_1 := \iota_1(\mathbb S^{n-1}) \) and \( \Sigma_2 := \iota_2(\mathbb S^{n-1}) \) are locally flat and disjoint. 
If \( \iota_1 \) and \( \iota_2 \) are compatibly oriented, then there exists an isotopy \( \vp \co \mathbb S^{n-1} \times [0,1] \to \mathbb S^n \) between \( \iota_1 \) and \( \iota_2 \) whose image is the locally flat annulus co-bounded by \( \Sigma_1 \) and \( \Sigma_2 \). 
\end{Lem}

\begin{proof}
Let \( A \) be the locally flat annulus co-bounded by \( \Sigma_1 \) and \( \Sigma_2 \). 
Choose an embedding \( \iota_A \co \mathbb S^{n-1} \times [0,1] \to \mathbb S^n \) whose image is \( A \) and such that \( \iota_A|_{\mathbb S^{n-1}\times\{0\}} = \iota_1 \).
Let \( \iota = \iota_A|_{\mathbb S^{n-1}\times\{1\}} \).
Then  \( \tau = \iota^{-1}\circ  \iota_2 \) is an orientation-preserving homeomorphism of \( \mathbb S^{n-1} \); hence, \( \tau \) is isotopic to the identity. 
Choose an isotopy \( H\co \mathbb S^{n-1} \times [0,1] \to \mathbb S^{n-1} \) such that \( H(x,0) = x \) and \( H(x,1) = \tau(x) \). 
Define \( h \co \mathbb S^{n-1} \times [0,1] \to \mathbb S^{n-1} \times [0,1] \) by \( h(x,t) = (H(x,t), t) \).
It is readily checked that \( h \) is a homeomorphism and that \( \vp:=  \iota_A\circ h \) is the desired map.
\end{proof}

We record as a corollary the main ways in which we will use the results above, each of which being a type of change of coordinates.

\begin{Cor}[Three change of coordinates principles]
\label{cor:coordinates}
Let \( n \in \bn \).
\begin{enumerate}[(1)]

\item Given any two locally flat \( (n-1) \)-dimensional spheres \( \Sigma \) and \( \Sigma' \) in \( \br^n \),  there exists a compactly supported ambient homeomorphism mapping \( \Sigma \) onto \( \Sigma' \).

\item Given two locally flat \( n \)-dimensional spheres \( \Sigma \) and \( \Sigma' \) in \( \mathbb S^n \times [0,1] \) each separating \( \mathbb S^n \times \{0\} \) from \( \mathbb S^n \times \{1\} \), there exists an ambient homeomorphism mapping \( \Sigma \) onto \( \Sigma' \) and fixing \( \mathbb S^n \times \{0,1\} \) pointwise. 

\item Let \( x, y \in \mathbb S^n \).
Given a sequence of locally flat annuli \( \{A_k\}_{k\in\bz} \) with pairwise-disjoint interiors  such that \( A_k \) and \( A_{k+1} \) share a boundary sphere and such that \( \mathbb S^n \ssm\{x,y\} = \bigcup_{k\in\bz} A_k \), there exists \( \tau \in \Homeo^+(\mathbb S^n) \) such that \( \tau(x) = x \), \( \tau(y) = y \), and \( \tau(A_k) = A_{k+1} \). \qed
\end{enumerate}
\end{Cor}

We can now introduce the notion of a topologically loxodromic homeomorphism. 
Tsuboi uses the terminology ``topologically hyperbolic homeomorphism'', but the definition given here is seemingly more stringent than the one he gives, and so we introduce a slight modification in our nomenclature to recognize this potential difference.

\begin{Def}
An orientation-preserving homeomorphism \( \tau \co \mathbb S^n \to \mathbb S^n \) is \emph{topologically loxodromic} if there exists \( \tau^+, \tau^- \in \mathbb S^n \) fixed by \( \tau \) and a sequence of  locally flat annuli \( \{ A_n\}_{n\in\bz} \) with pairwise-disjoint interiors such that 
\begin{enumerate}[(i)]
\item \( \mathbb S^n \ssm \{\tau_\pm\} = \bigcup_{k\in\bz} A_k \),
\item \( A_k \) shares a boundary component with \( A_{k+1} \), 
\item \( \tau(A_k) = A_{k+1} \), and
\item \( \lim_{k\to\pm\infty} \tau^k(x) = \tau^\pm \) for all \( x \in \mathbb S^n \ssm \{\tau^+, \tau^-\} \).
\end{enumerate}
We the say the sequence of annuli \( \{A_n\}_{n\in\bn} \) is \emph{suited} to \( \tau \), and we call \( \tau^+ \) the \emph{sink} of \( \tau \) and \( \tau^- \) the \emph{source}. 
Note that every homeomorphism of \( \br^n \) and \( \mathbb S^{n-1} \times \br \) can be viewed as a homeomorphism of \( \mathbb S^n \) that fixes one or two points, respectively, corresponding to the ends, and so we say a homeomorphism of either of these spaces is topologically loxodromic if it is a topologically loxodromic homeomorphism as a homeomorphism of \( \mathbb S^n \). 
\end{Def}

The main motivating examples of topologically loxodromic homeomorphisms are loxodromic M\"obius transformations of \( \mathbb S^n \),  dilations of \( \mathbb R^n \), and translations of \( \mathbb S^n \times \br \) of the form \( (x,t) \mapsto (x, t+t_0) \) for some \( t_0 \in \br \).  
Of course the latter two examples are just loxodromic M\"obius transformations in different coordinates.

In the next proposition, we prove that any two topologically loxodromic homeomorphisms of \( \mathbb S^n \) are conjugate.  
This is the key fact in Tsuboi's argument, and the following results appear as part of Tsuboi's proof in reference to particular maps.
We pull the ideas out here and formalize them in the general setting. 
As a corollary, we provide the analogous statement for \( \br^n \) and \( \mathbb S^n \times \br \).

\begin{Prop}
\label{prop:conjugate}
Any two topologically loxodromic homeomorphisms in \( \Homeo_0(\mathbb S^n) \) are conjugate in \( \Homeo_0(\mathbb S^n) \). 
\end{Prop}

\begin{proof}
Let \( \tau, \sigma \in \Homeo_0(\mathbb S^n) \) be topologically loxodromic. 
Let \( \{ A_n\}_{n\in\bz} \) and \( \{ B_n\}_{n\in\bz} \) be sequences of locally flat annuli suited to \( \tau \) and \( \sigma \), respectively. 
Let \( \Sigma_A \) be the component of \( A_0 \) such that \( \partial A_0 = \Sigma_A \cup \tau(\Sigma_A) \), and similarly, define \( \Sigma_B \) so that \( \partial B_0 = \Sigma_B \cup \sigma(\Sigma_B) \). 

Fix embeddings \( \iota_A, \iota_B \co \mathbb S^{n-1} \to \mathbb S^n \) such that the images of \( \iota_A \) and \( \iota_B \) are \( \Sigma_A \) and \( \Sigma_B \), respectively. 
Applying \Cref{lem:isotopy}, we obtain embeddings \( \vp_A, \vp_B \co \mathbb S^{n-1}\times [0,1] \to \mathbb S^n \) whose images are \( A_0 \) and \( B_0 \), respectively, and such that \( \vp_A \) (resp., \( \vp_B \)) is an isotopy between \( \iota_A \) and \( \tau\circ \iota_A \) (resp., \( \iota_B \) and \( \sigma \circ \iota_B \)).  
Set \( \vp = \vp_A \circ \vp_B^{-1} \).

Define \( h \co \mathbb S^n \to \mathbb S^n \) by \( h(\sigma_\pm) = \tau_\pm \) and \( h(x) = (\tau^n \circ \vp \circ \sigma^{-n}) (x) \) for \( x \in B_n \).
It is now readily checked that \( h \) is well-defined for the elements of \( B_n \cap B_{n+1} \), establishing that \( h \) is a well-defined orientation-preserving homeomorphism of \( \mathbb S^n \).
Now, for \( x \in A_n \),
\begin{align*}
(h\circ \sigma \circ h^{-1})(x) &= \left[(\tau^{n+1}\circ \vp \circ \sigma^{-(n+1)})\circ \sigma \circ (\sigma^n \circ \vp^{-1} \circ \tau^{-n})\right](x)\\
	&= \tau(x).
\end{align*}
Hence, \( \tau = h \circ \sigma \circ h^{-1} \).
\end{proof}

\begin{Cor}
\label{cor:conjugate}
Let \( G \) denote either \( \Homeo_0(\br^n) \) or \( \Homeo_0(\mathbb S^n \times \br) \).
Any two topologically loxodromic homeomorphisms in \( G \) that share a sink or a source are conjugate in \( G \).
\end{Cor}

\begin{proof}
Viewing two given topologically loxodromic homeomorphisms in \( G \) as homeomorphisms of \( \mathbb S^n \) with a shared sink, a shared source, or both, the conjugating map constructed in \Cref{prop:conjugate} will preserve any shared source or sink, and hence the conjugation occurs in the appropriate group.
\end{proof}

\section{Proofs}

We prove the theorem first for annuli then for Euclidean spaces.  

\begin{Thm}
\label{thm:annulus}
For \( n \in \bn \), every element of \( \Homeo_0(\mathbb S^{n} \times \br) \) is a commutator. 
\end{Thm}

\begin{proof}
Let \( f \in \Homeo_0(\mathbb S^{n} \times \br) \). 
Let \( \Sigma_0 = \mathbb S^{n} \times \{0\} \), and set \( t_0 = 0 \). 
Let \( n_0 \in \bn \) such that \( A_0 := \mathbb S^{n} \times [-n_0, n_0] \) contains \( \Sigma_0 \cup f(\Sigma_0) \) in its interior. 
Now, choose \( t_1, n_1 \in \bn \) such that \( \Sigma_{1} = \mathbb S^{n} \times \{t_1\} \)  satisfies \( \Sigma_{1} \cup f(\Sigma_{1}) \) is contained in the interior of \( A_{1} := \mathbb S^{n} \times [n_0, n_1] \). 
Continuing in this fashion in both directions, we construct sequences \( \{n_k\}_{k\in\bz} \) and \( \{t_k\}_{k\in\bz} \) of integers such that, by setting \( \Sigma_k =  \mathbb S^{n} \times \{t_k\} \) and \( A_k = \mathbb S^{n} \times [n_k, n_{k+1}] \), we have \( \Sigma_k \cup f(\Sigma_k) \) is contained in the interior of \( A_k \). 

Now, let \( g' \in \Homeo_0(\mathbb S^{n} \times \br) \) such that \( g'(A_k) = A_{k+1} \).
Note, by construction, \( \bigcup_{k\in\bz} A_k \) is all of \( \mathbb S^n \times \br \), and so \( g' \) is topologically loxodromic. 
As both \( (g'\circ f)(\Sigma_k) \) and \( \Sigma_{k+1} \) are locally flat annuli contained in the interior of \( A_{k+1} \), we can choose \( h_k \in \Homeo_0(\mathbb S^{n} \times \br) \) that is supported in the interior of \( A_{k+1} \) and satisfies \( h_k(g'(f(\Sigma_k))) = \Sigma_{k+1} \). 
The sequence \( \{h_k\}_{k\in \bz} \) consists of homeomorphisms with pairwise-disjoint supports, and hence, we can define \( h = \prod_{k\in\bz} h_k \).
Set \( g = h\circ g' \). 

As the support of each \( h_k \) is contained in the interior of \( A_{k+1} \), we have that \[ g(A_k) = h_k(g'(A_k)) = h_k(A_{k+1}) = A_{k+1}; \] in particular, \( g \) is topologically loxodromic.
Let \( B_k \) be the locally flat annulus co-bounded by \( \Sigma_k \) and \( \Sigma_{k+1} \).
Then, as \( (g\circ f)(\Sigma_k) = \Sigma_{k+1} \), we have \( (g\circ f)(B_k) = B_{k+1} \).
Therefore, \( g\circ f \) is topologically loxodromic.
Moreover, \( g \) and \( g\circ f \) share the same sink, and hence they are  conjugate  by \Cref{cor:conjugate}.
This establishes that \( f \) can be expressed as a commutator.  
\end{proof}

\begin{Thm}
\label{thm:rn}
For \( n \in \bn \), every element of \( \Homeo_0(\mathbb R^n) \) is a commutator. 
\end{Thm}

\begin{proof}
Letting \( \mathbb S^0 \) be a singleton, the proof  in the annulus case with \( n = 0 \) shows that every element of \( \Homeo_0(\br) \) can be expressed as a commutator. 
We may now assume that \( n > 1 \).

Fix \( z \in \mathbb S^n \), and identify \( \Homeo_0(\br^n) \) with the stabilizer of \( z \) in \( \Homeo_0( \mathbb S^n) \).  
Let \( f \) be an element of \( \Homeo_0(\mathbb S^n) \) such that \( f(z) = z \). 
We will show that \( f \) can be expressed as a commutator of a pair of elements in \( \Homeo_0(\mathbb S^n) \) that each fix \( z \). 
The identity homeomorphism is clearly a commutator, so we may assume that \( f \) is not the identity, and therefore there exists \( y \in \mathbb S^n \) such that \( f(y) \neq y \). 
By continuity, there exists a locally flat ball \( D_0 \) centered at \( y \) such that \( f(D_0) \cap D_0 = \varnothing \). 
Choose a locally flat ball \( D_1 \) centered at \( z \) that is disjoint from \( D_0 \cup f(D_0) \). 
Again by continuity, by shrinking \( D_1 \) if necessary, we may assume that \( f(D_1) \) is also disjoint from \( D_0 \cup f(D_0) \). 

First, in \( D_1 \), we proceed identically as we did in the annulus case. 
Choose a locally flat  annulus \( A_1 \) such that \( \partial D_1 \cup f(\partial D_1) \) is contained in the interior of \( A_1 \) and such that \( A_1 \) is disjoint from \( D_0 \cup f(D_0) \). 
Now choose a locally flat ball \( D_{2} \) centered at \( z \) such that \( D_2 \cup f(D_2) \) is disjoint from \( A_1 \). 
Choose a locally flat  annulus \( A_2 \) such that \( A_1 \) and \( A_2 \) share a boundary component and \( \partial D_2 \cup f(\partial D_2) \) is contained in the interior of \( A_2 \). 
Continuing in this fashion, we build a sequence of locally flat  annuli \( \{ A_k\}_{k\in\bn} \) and locally flat balls \( \{D_k\}_{k\in\bn} \) such that \( A_k \) and \( A_{k+1} \) share a boundary, \( \partial D_k \cup f(\partial D_k) \) is contained in the interior of \( A_k \), and \( \bigcap_{k\in\bn} D_k = \{ z \} \). 
Note that the last condition is not a priori guaranteed, but at each stage we are free to choose \( D_k \) to have radius less than \( 1/k \), forcing the intersection of the \( D_k \) to be \( \{z\} \). 

Let \( g_0' \) be a topologically loxodromic homeomorphism of \( \mathbb S^n \) that fixes \( z \), that maps \( A_{k+1} \) to \( A_{k} \) for \( k \in \bn \), and that maps \( f(\partial D_1) \) to \( \partial D_0 \). 
For \( k \leq 0 \), let \( A_k = (g_0')^{\circ (n+1)}(A_1) \). 
Then, \( (g_0'\circ f)(\partial D_{k+1}) \) and \( \partial D_{k} \) are locally flat spheres in the annulus \( A_k \), and so we may choose  \( h_k \in \Homeo( \mathbb S^n ) \) supported in the interior of \( A_{k} \) such that \( h_k(g_0'(f(\partial D_{k+1}))) = \partial D_{k} \). 
As the \( h_k \) have pairwise disjoint support, we can define \( h = \prod_{k\in\bn} h_k \). 
Let \( g_0 = h \circ g_0' \). 
Then, for \( k \in \bn\cup\{0\} \), setting \( T_{k} \) to be the locally flat  annulus bounded by \( \partial D_{k} \) and \( \partial D_{k+1} \), we have \( (g_0\circ f)(T_{k+1}) = T_{k} \). 

At this point, \( g_0 \circ f \) behaves like a topologically loxodromic homeomorphism when restricted to \( D_1 \), and so now we have to edit \( g_0 \) in the complement of \( D_1 \). 
This portion of the argument is a version of Tsuboi's argument for spheres. 
Let \( B_0 = g_0(f(D_0)) \), and note that \( f(B_0) \subset f(D_0) \) and \( g_0(f(B_0)) \subset B_0 \).
Moreover, observe that \( B_0 \cup f(D_0) \subset  A_0 \); hence, modifying \( g_0 \) in either \( f(D_0) \) or \( B_0 \) will not change the fact that \( g_0 \) is topologically loxodromic. 

Fix \( x \in B_0 \).
We may assume that \( g_0(f(x)) = x \); indeed, if not, then we may post compose \( g_0 \) with a homeomorphism of \( \mathbb S^n \) supported in \( B_0 \) that maps \( g_0(f(x)) \) to \( x \). 
And, as just noted above, this edited version of \( g_0 \) remains topologically loxodromic. 
We will now recursively edit \( g_0 \)  in \( f(D_0) \) so that \( g_0 \circ f \) will be topologically loxodromic with \( z \) its sink and \( x \) its source.

Below, given a subset \( X \), we let \( X^\mathrm o \) denote its interior. 
Let \( \Sigma_0 = \partial B_0 \), let \( B_0' = f(B_0) \), let \( B_1 = g_0(B_0') \), and let \( \Sigma_1 = \partial B_1  \).
Note \( \Sigma_0 \cap \Sigma_1 = \varnothing \), and therefore \( \Sigma_0 \) and \( \Sigma_1 \) co-bound a locally flat  annulus, which we label \( T_{-1} \). 
Then, by continuity, we can choose a locally flat ball \( B_1' \) in the interior of \( B_0' \) centered at \( f(x) \) such that  \( g_0(B_1') \subset B_1^\mathrm o \).  
We can then choose \( \sigma_1 \in \Homeo(\mathbb S^n) \) such that \( \sigma_1 \) is supported in \( f(B_0) \) and such that \( \sigma_1(f(\Sigma_1)) = \partial B_1' \). 
Set \( B_2 = g_0(B_1') \), and  set \( \Sigma_2 = \partial B_2 \).
Now, set \( T_{-2} \) to be the  locally flat annulus co-bounded by \( \Sigma_1 \) and \( \Sigma_2 \), and set  \( g_1 = g_0 \circ \sigma_1 \).
Then, \( (g_1\circ f)(T_k) = T_{k-1} \) for all \( k \in \bn \cup \{0, -1\} \). 
Note that \( B_1' \) can be chosen arbitrarily small.

Continuing in this fashion, for each \( m \in \bn \), we obtain a  locally flat annulus \( T_{-m} \) and a homeomorphism \( \sigma_m \) of \( \mathbb S^n \) supported in a ball of radius \( 1/m \) centered at \( f(x) \) such that
\begin{enumerate}[(1)]
\item \( \mathbb S^n \ssm \{ x,z\} = \bigcup_{k\in\bz} T_k \),
\item \( (g_m\circ f)(T_k) = T_{k-1} \) for every integer \( k \geq -m \), and
\item \( g_m = g_{m-1}\circ \sigma_m = g_0 \circ (\sigma_1 \circ \sigma_{2} \circ \cdots \circ \sigma_m) \).
\end{enumerate}
These properties guarantee that \( \lim_{m\to\infty} g_m \) exists, call it \( g \), and that \( (g \circ f)(T_k)= T_{k-1} \) for every \( k \in \bz \); hence, \( g\circ f \) is topologically loxodromic. 
Moreover, \( g \) is topologically loxodromic as \( g \) agrees with \( g_0 \) outside of \( B_0 \cup f(D_0) \); in particular, \( g(A_k) = g_0(A_k) = A_{k+1} \) for all \( k \in \bz \).  
And, as \( g \) and \( g\circ f \) share the same sink, they are conjugate in \( \Homeo_0(\br^n) \) by \Cref{cor:conjugate}.
Thus, \( f \) can be expressed as a commutator in \( \Homeo_0(\br^n) \).
\end{proof}

We finish by providing a proof of \Cref{cor:2}.

\begin{proof}[Proof of \Cref{cor:2}]
Without loss of generality, it is enough to prove the statement with \( p_i \in \bn \) for each \( i \in \{1, \ldots, r\} \), as we can always replace an element with its inverse in the decomposition to switch the sign of the exponent. 

Fix \( g \in G \).
We have shown in the above proofs that there exist \( f, h \in G \) such that \( g = f\circ h \) and both \( f \) and \( h \) are topologically loxodromic. 
Applying the same techniques to \( f \), we can write \( f = f_1 \circ f_2 \) with both \( f_1 \) and \( f_2 \) topologically loxodromic. 
Therefore, continuing this splitting as many times as necessary, we can write \( g = \prod_{i=1}^r f_i \) with \( f_i \) topologically loxodromic. 
Now, every power of a topologically loxodromic homeomorphism is itself topologically loxodromic. 
Therefore, by \Cref{prop:conjugate} and \Cref{cor:conjugate}, \( f_i^{\circ p_i} \) is conjugate to \( f_i \). 
Choose \( h_i \) such that \( f_i = h_i \circ f_i^{\circ p_i} \circ h_i^{-1} \). 
Then, setting \( g_i  = h_i \circ f_i \circ h_i^{-1} \), we have \( f_i = g_i^{\circ p_i} \), and hence, \( g = \prod_{i=1}^r g_i^{\circ p_i} \). 
\end{proof}

\bibliographystyle{amsalpha}
\bibliography{commutators}

\end{document}